\documentclass[11pt, reqno]{amsart}
\usepackage[utf8]{inputenc}
\usepackage{amsmath,bm}
\usepackage{amsthm}
\usepackage{amssymb}
\usepackage{amsfonts}
\usepackage{MnSymbol}
\usepackage{tcolorbox}
\usepackage{blindtext,amsmath,amsthm,amsfonts, textcomp, mathrsfs, mathtools}
\usepackage{relsize}
\usepackage[shortlabels]{enumitem}

\usepackage[top=27mm,bottom=27mm,left=27mm,right=27mm]{geometry}
\newtheorem{theorem}{Theorem}[section]
\newtheorem{conjecture}[theorem]{Conjecture}
\newtheorem{corollary}[theorem]{Corollary}

\newtheorem*{theorem*}{Theorem}
\newtheorem*{remark*}{Remark}
\newtheorem*{problem*}{Problem}
\newtheorem*{conjecture*}{Conjecture}
\newtheorem*{question*}{Question}
\newtheorem{lemma}[theorem]{Lemma}

\newcommand{\rom}[1]{\uppercase\expandafter{\romannumeral #1\relax}}
\newcommand{\Q}{\mathbb{Q}}
\newcommand{\Z}{\mathbb{Z}}
\newcommand{\rL}{\mathcal{L}}
\newcommand{\rK}{\mathcal{K}}
\newcommand{\rO}{\mathcal{O}}
\newcommand{\R}{\mathbb{R}}
\newcommand{\C}{\mathbb{C}}
\newcommand{\rp}{\mathfrak{p}}
\newcommand{\q}{\mathfrak{q}}
\def\house#1{{%
    \setbox0=\hbox{$#1$}
    \vrule height \dimexpr\ht0+1.4pt width .4pt depth \dp0\relax
    \vrule height \dimexpr\ht0+1.4pt width \dimexpr\wd0+2pt depth \dimexpr-\ht0-1pt\relax
    \llap{$#1$\kern1pt}
    \vrule height \dimexpr\ht0+1.4pt width .4pt depth \dp0\relax
}}

\begin{document}

\title[Lower bounds on height for non-Galois extensions]{Lower bound on height of algebraic numbers and \\ low lying zeros of the Dedekind zeta-function}

\author[Anup B. Dixit]{Anup B. Dixit}
\author[Sushant Kala]{Sushant Kala}

\address{Department of Mathematics\\ Institute of Mathematical Sciences (HBNI)\\ CIT Campus, IV Cross Road\\ Chennai\\ India-600113}
\email{anupdixit@imsc.res.in}

\address{Department of Mathematics\\ Institute of Mathematical Sciences (HBNI)\\ CIT Campus, IV Cross Road\\ Chennai\\ India-600113}
\email{sushant@imsc.res.in}
\date{}

\begin{abstract}
In this paper, we establish lower bounds on Weil height of algebraic integers in terms of the low lying zeros of the Dedekind zeta-function. As a result, we prove Lehmer's conjecture for certain infinite non-Galois extensions conditional on GRH. We also introduce and study a condition on prime ideals with small norms for arbitrary infinite extensions, in the spirit of a prime splitting condition for infinite Galois extensions introduced by E. Bombieri and U. Zannier.
\end{abstract}



\maketitle









\section{\bf Introduction}
\medskip

Let $K$ be a number field. For $\alpha \in K^*$, the absolute logarithmic height or logarithmic Weil height is defined as
\begin{equation*}
    h(\alpha) = \sum_{v \in M_K} \log^+ |\alpha|_{v},
\end{equation*}
where $M_K$ is the set of places of $K$, $\log^+ x = \max( 0, \log x)$ and $|\alpha|_{v}$ is the normalized valuation on $\alpha$ defined as follows. For non-archimedean valuation $v$,
\begin{equation*}
    |\alpha|_{v} := \left( N\mathfrak{p}\right)^{- \frac{ord_{\mathfrak{p}}(\alpha)}{[K \,: \, \Q]}},
\end{equation*}
where $\mathfrak{p}$ is the prime ideal corresponding to the valuation $v$. If $v$ is archimedean, then
$$
    |\alpha|_{v} := | \sigma(\alpha)|^{\frac{1}{[K \,: \,\Q]}},
$$
where $\sigma$ is the embedding of $\alpha$ corresponding to $v$.
It is an important theme in Diophantine geometry to study lower bounds on $h(\alpha)$. By a well-known theorem of Kronecker \cite{Kronecker}, $h(\alpha)=0$ if and only if $\alpha$ is a root of unity. In 1933, Lehmer \cite{lehmer} conjectured that if a non-zero $\alpha\in \overline{\Q}$ is not a root of unity, then
\begin{equation*}
    h(\alpha) \geq \frac{c}{[\Q(\alpha):\Q]}
\end{equation*}
for an absolute constant $c>0$.  \\

One could also reformulate Lehmer's conjecture in terms of the Mahler measure as follows. Let $f(x) = a_n x^n + \cdots + a_1 x + a_0 \in \Z[x]$ be the minimal polynomial of $\alpha$. Then, its Mahler measure is defined as
\begin{equation*}
    M(\alpha) := |a_n| \prod_{i} \max(1,|\alpha_i|),
\end{equation*}
where $\alpha_i$'s denote the conjugates of $\alpha$. Mahler measure is connected to the height of $\alpha$ by the relation
\begin{equation*}
    \log M(\alpha) =  h(\alpha) \, [\Q(\alpha): \Q].
\end{equation*}
Thus, Lehmer's conjecture can be interpreted as
\begin{equation*}
    \log M(\alpha) \geq c
\end{equation*}
for non-zero $\alpha \in \overline{\Q}$ not a root of unity and $c>0$ an absolute constant. Although this conjecture remains open, significant progress has been made in recent times. Using a sharpened version of Siegel's lemma and thereby constructing an auxiliary polynomial with small coefficients, Dobrowolski \cite{Dob} proved that for non-zero $\alpha$ not a root of unity,
\begin{equation*}
   \log M(\alpha) \geq c \left(\frac{\log\log n}{\log n}\right)^3,
\end{equation*}
where $n=[\Q(\alpha) : \Q]$ and $c>0$ is an absolute constant. Subsequently, this constant $c$ has been improved in the works of U. Rausch \cite{Rausch} and P. Voutier \cite{Voutier}. In 1971, C. Smyth \cite{smyth} showed that Lehmer's conjecture holds for all non-reciprocal algebraic numbers. We call $\alpha\in \overline{\Q}$ reciprocal if $\alpha$ and $1/\alpha$ are conjugates. This result, with a weaker constant, was also obtained by R. Breusch \cite{Breusch} in 1951. On another front, suppose $\alpha\in \overline{\Q}$ is such that the Galois closure of $\Q(\alpha)/\Q$, say $K_{\alpha}$ has relatively smaller degree, i.e., $[K_{\alpha} : \Q]$ is polynomial in $[\Q(\alpha):\Q]$, then F. Amoroso and S. David \cite{David} proved that Lehmer conjecture holds for all such $\alpha$.\\ 

\noindent
A weaker version of Lehmer's conjecture was proposed by A. Schinzel and H. Zassenhaus \cite{SZ}, which was recently resolved by V. Dimitrov \cite{Dimitrov}. Let $\house{\strut \alpha}$ denote the house of $\alpha \in \overline{\Q}$ defined as
$$
   \house{\strut \alpha} = \max_i \{|\alpha_i|\},
$$
where $\alpha_i$'s are the conjugates of $\alpha$. Then, Dimitrov's theorem states that for any non-zero $\alpha\in \overline{\Q}$ not a root of unity,
\begin{equation*}
    \log \house{\strut\alpha} > \frac{c}{[\Q(\alpha): \Q]}
\end{equation*}
for $c = (\log 2)/4$. It is easily seen that Lehmer's conjecture implies Dimitrov's theorem, albeit for a different constant $c>0$. The reader may refer to the excellent survey articles \cite{Smyth} and \cite{Verger-Gaugry} for a comprehensive account of this problem.\\
\medskip

In this paper, we first obtain a lower bound for the height of a non-zero algebraic integer $\alpha$, where the associated Dedekind-zeta function of $\Q(\alpha)$ has an abundance of low-lying zeros.\\

Let $K/\Q$ be a number field, $n_K$ denote its degree $[K:\Q]$ and $d_K$ its absolute discriminant $|disc(K/\Q)|$ respectively. The Dedekind zeta-function $\zeta_K(s)$ associated to $K$ is defined on the half-plane $\Re(s)>1$ as
\begin{equation*}
    \zeta_K(s) := \sum_{\mathfrak{a} \subset \mathcal{O}_K} \frac{1}{N\mathfrak{a}^s},
\end{equation*}
where $\mathfrak{a}$ runs over all non-zero integral ideals of the ring of integers $\mathcal{O}_K$. The function $\zeta_K(s)$ has an analytic continuation to the whole complex plane except for a simple pole at $s=1$. The generalized Riemann hypothesis (GRH) states that all the non-trivial zeros of $\zeta_K(s)$ lie on $\Re(s)=1/2$.\\ 

Let $N_K(T)$ denote the number of zeros of $\zeta_K(s)$ in the critical strip $0\leq \Re(s) \leq 1$ with $|\Im(s)| < T $ and define 
$$
\lambda_K(T):= \sum_{\substack{|t|< T \\ \zeta_K(\sigma + it)=0}}  \frac{1}{1+t^2} ,
$$
where $t$ runs over the imaginary part of the non-trivial zeros of $\zeta_K(s)$. We show the following.
\begin{theorem}\label{Lehmer-GRH}
    Let $\alpha$ be a non-zero algebraic integer and  $K=\Q(\alpha)$. Under GRH 
    $$
    2 \, n_K \, \, h(\alpha) \geq \frac{3.67}{n_K} \left( \lambda_K(2) - \frac{1}{5} N_K(2) \right) - \log n_K + O(1),
    $$
    where the implied constant is absolute.
\end{theorem}

\medskip

Since $\lambda_K(T) $ is an increasing function on $[0,2]$, we have
$$
    \lambda_K(2)-\frac{1}{5}N_K(2) \geq \lambda_K(1)-\frac{1}{5}N_K(1) \geq \frac{3}{10} N_K(1).
$$
Now, it is easy to deduce from Theorem \ref{Lehmer-GRH} that, under GRH, for the set of all $\alpha\in \overline{\Q}$ with ${N_K(1) \geq n_K \log n_K}$ satisfy
$$
    h(\alpha) \geq 0.55 \, \frac{\log n_K}{n_K} + O(1).
$$
Therefore, this set of $\alpha$'s satisfies Lehmer's conjecture. Moreover, for the set of  $\alpha$ with $N_K(1) \geq n_K^2$, the above computation implies an absolute lower bound for $h(\alpha)$ independent of the degree. In other words, abundance of low-lying zeros of the Dedekind zeta-function leads to lower bounds on the height of an algebraic number. \\

The low-lying zeros of $\zeta_K(s)$ is a purely analytic data. Our next goal is to obtain lower bounds on $h(\alpha)$ in terms of arithmetic data, i.e., in terms of prime ideals with small norms. For Galois extensions $\Q(\alpha)/\Q$, this is captured by the splitting nature of primes, which has been exploited in the works of E. Bombieri and U. Zannier \cite{BZ} and F. Amoroso and S. David \cite{David}. We show that if several prime ideals in $\Q(\alpha)$ have small norms, then under GRH, Lehmer's conjecture holds for all such $\alpha$. This can be compared with the result of M. Mignotte \cite{mignote}.\\ 

Let $\pi(x)$ denote the number of primes less than $x$. For a number field $K/\Q$ and a rational prime power $q$, define
\begin{equation*}
    \mathcal{N}_q(K) := \text{ the number of prime ideals of } K \text{ with norm } q.
\end{equation*}

\begin{theorem} \label{Uncondional}
    Let $\delta, \epsilon > 0$ be fixed. For a non-zero algebraic integer $\alpha$, let $K=\Q(\alpha)$. We say that $\alpha \in S$, if for some $Y$ satisfying
    $$
    (\log n_K)^2 < Y < {n_K}^{\frac{1}{2}},
    $$
    we have $\mathcal{N}_p(K) >  \delta n_K$ for at least $\epsilon \, \pi(Y)$ number of primes $p \leq Y$. Then under GRH,
    $$
        \liminf_{\alpha \in S} \, h(\alpha) \, n_K = \infty.
    $$
    In particular, Lehmer's conjecture holds for $S$.
\end{theorem}

\medskip

The paper is organized as follows. In Section \ref{section-2}, we introduce certain invariants for infinite extensions $\rL/\Q$, not necessarily Galois and formulate a conjecture analogous to Bombieri-Zannier's result in \cite{BZ}. We also state a result on the low-lying zeros for infinite Galois extensions. In Section \ref{section-3}, we discuss the connection between lower bounds on discriminant of a number field and the low-lying zeros on Dedekind zeta-function. We devote Section \ref{section-4} to setting up the preliminaries and proving the necessary lemmata and in Section \ref{section-5} we give proofs of all the main theorems.
\bigskip

\section{\bf Asymptotically positive extensions}\label{section-2}
\medskip

A well-known theorem of Northcott \cite{Northcott} states that there are only finitely many algebraic numbers with a bounded degree and height. This inspires the definition of Northcott property (N). We say that a set of algebraic numbers $S$ has property (N) if for any real number $\epsilon >0$,
\begin{equation*}
    \{\alpha \in S : h(\alpha) < \epsilon\}
\end{equation*}
is a finite set. By Northcott's theorem, the set of all $\alpha \in \overline{\Q}$ with degree $\leq d$ satisfies property (N). Thus, it is interesting to study property (N) for infinite extensions of $\Q$. Another related property on a set of algebraic numbers is the Bogomolov property (B). We say a set $S \subset \overline{\Q}$ satisfies property (B) if there exists a constant $c > 0$, such that 
\begin{equation*}
    \{\alpha \in S, \alpha\text{ non-zero and not a root of unity}: h(\alpha) < c\}
\end{equation*}
is an empty set. If a set $S$ satisfies property (N), then it clearly satisfies property (B). Moreover, if $S$ satisfies property (B), then there is an absolute lower bound for $h(\alpha)$ with $\alpha \in S$ not a root of unity, and hence Lehmer's conjecture holds for $S$. \\

In 2000, F. Amoroso and R. Dvornicich  \cite{Amoroso} showed that $\Q_{ab}$, the maximal abelian extension of $\Q$ satisfies property (B). In fact, they proved that for non-zero $\alpha \in \Q_{ab}$, not a root of unity, 
$$
h(\alpha) \geq \frac{\log 5}{12}.
$$ 
Soon after, E. Bombieri and U. Zannier \cite{BZ} obtained a sufficient condition for infinite Galois extensions over $\Q$ to satisfy property (B). They produced explicit lower bounds on $h(\alpha)$ in terms of the prime splitting behavior. More precisely, let $\rL$ be an infinite Galois extension of $\Q$.
Define
\begin{equation*}
    S(\rL):= \{p \,\text{ prime} \, : \, [\rL_{v} : \Q_p] < \infty \,\, \text{for some}\,\, v \,\, \text{above}\,\, p\}.
\end{equation*}
Then, they showed that
\begin{equation}\label{B-Z}
    \liminf_{\alpha\in \rL} \,\,  h(\alpha) \, \,  \geq \frac{1}{2} \sum_{p\in S(\rL)} \frac{\log p}{e_p \, \left( p^{f_p} + 1 \right)},
\end{equation}
where $e_p$ is the ramification index and $f_p$ is the residual degree of $\rL_{v}/\Q_p$. Thus, if $S(\rL)$ is non-empty, then $\rL$ satisfies property (B) and hence Lehmer's conjecture is true for $\rL$. Furthermore, if the right-hand side in \eqref{B-Z} diverges, then $\rL$ satisfies property (N). 
\medskip

Note that for $\rL / \Q$ Galois, the condition $[\rL_{v} : \Q_p] <\infty$ is equivalent to saying that there exists a finite Galois extension $L/ \Q$ with $L\subset \rL$ such that all the prime ideals above $p$ in $L$ split completely in $\rL$. However, it is difficult to obtain similar results for non-Galois extensions, where we do not have eventual complete splitting. Naively taking the Galois closure and applying \eqref{B-Z} often leads to poor bounds on $h(\alpha)$. This motivates us to independently consider non-Galois extensions using analytic techniques. More precisely, we initiate the study of this problem by connecting it to the low-lying zeros of Dedekind zeta-function.\\

\medskip

We begin by introducing notation and proving some basic lemmata. An infinite extension $\rL/\Q$ can be written as a tower of number fields
$$
\rL \supsetneq \cdots \supsetneq L_m \supsetneq L_{m-1} \supsetneq \cdots \supsetneq L_1 = \Q,
$$
where $L_i/\Q$ are finite extensions. For any rational prime power $q = p^k$, define
$$
    \psi_q(\rL) = \psi_q := \lim_{i\to \infty} \frac{\mathcal{N}_q(L_i)}{n_{L_i}},
$$
where
$\mathcal{N}_q(L)$ is the number of prime ideals in $L$ with norm $q$. This limit is well-defined and in fact the following lemma proves a stronger result.

\begin{lemma}\label{lemma-1}
Let $L/K$ be an extension of number fields. Then,
$$
    \sum_{q \leq x} \frac{\mathcal{N}_q(K)\, \log q}{[K:\Q]} \geq \sum_{q \leq x} \frac{\mathcal{N}_q(L)\, \log q}{[L:\Q]} ,
$$
where $q$ runs over all prime powers.
\end{lemma}
\begin{proof}
If a prime ideal $\q$ in $L$ has norm $\leq x$, then the prime ideal below, $\rp = \q \cap K$ also has norm $\leq x$. Suppose a prime ideal $\rp$ in $K$  splits into $\{\q_1, \q_2, \cdots, \q_r\}$ in $L$. Then, 
$$
    \prod_i N(\q_i) \leq N(\rp)^{[L \, : \, K]}.
$$
Therefore,
$$
    \sum_{m=1}^n m \,\, \mathcal{N}_{p^m}(L) \leq [L : K] \sum_{m=1}^n m \,\,  \mathcal{N}_{p^m}(K). 
$$
Mutliplying by $\log p$ and dividing by $[L:\Q]$, we obtain for any prime $p$
$$
    \sum_{q\leq x, q=p^k}  \frac{\mathcal{N}_q(K)\, \log q}{[K:\Q]} \geq \sum_{q \leq x, q= p^k} \frac{\mathcal{N}_q(L)\, \log q}{[L:\Q]}.
$$
Summing over all primes $p\leq x$, we obtain the lemma.

\end{proof}
\medskip

Hence, we deduce that for any tower $\rL =\{L_i\}$ and any $x > 1$, the limit 
$$
\lim_{i\to\infty} \sum_{q\leq x, q= p^k} \frac{\mathcal{N}_q(L_i)\, \log q}{[L_i:\Q]}
$$
exists. Therefore, inductively, we can conclude that the limit
$$
\psi_q = \lim_{i\to \infty} \frac{\mathcal{N}_q(L_i)}{[L_i:\Q]},
$$
is well-defined for all prime powers $q$ and takes value between $0$ and $1$.\\

We say that an infinite extension $\rL/\Q$ is \textit{asymptotically positive} if
$$
    \psi_q > 0
$$
for some prime power $q$. Under this notation, for an infinite Galois extension $\rL/\Q$, the RHS of \eqref{B-Z} can be written as
\begin{equation}\label{Bom-Zan-sum}
\frac{1}{2} \sum_{p\in S(\rL)} \frac{\log p}{e_p \, \left( p^{f_p} + 1 \right)} = \frac{1}{2}\sum_q \psi_q\, \frac{\log q}{q+1},
\end{equation}
where $q$ runs over all prime powers. This interpretation offers a remarkable advantage as it allows us to define analogous quantities for non-Galois extensions. Moreover, the condition $S(\rL)$ being non-empty is equivalent to $\rL/\Q$ being asymptotically positive.
\medskip

As a natural generalization of \eqref{B-Z}, it is reasonable to expect the following for infinite extensions over $\Q$, which are not necessarily Galois.
\begin{conjecture}
    Let $\rK / \Q$ be an infinite extension of $\Q$. Then,
    $$
    \liminf_{\alpha \in \rK} h(\alpha) \geq \frac{1}{2} \sum_{q} \psi_q \, \frac{\log q}{q+1},
    $$
    where $q$ runs over all prime powers.
\end{conjecture}

\medskip 

For infinite Galois extensions, we study the effect of the extension being asymptotically positive on the low-lying zeros of the Dedekind zeta-functions and the minimal index of number fields. For a non-zero algebraic integer $\alpha$, say $K=\Q(\alpha)$, let $I(\alpha):= [\rO_K :\Z[\alpha]]$ denote the index of $\alpha$, $d(\alpha)$ be the discriminant of the minimal polynomial of $\alpha$ and $d_K$ be the absolute discriminant of the number field $|disc(K/\Q)|$. It is easy to see that
\begin{equation}\label{index-disc}
    d(\alpha) = I(\alpha)^2 \, d_K.
\end{equation}
For any number field $K$, the minimal index is defined as 
$$
I_K = \min\{I(\alpha): \alpha \in \rO_K \text{ and }K=\Q(\alpha)\}. 
$$

\noindent
We will show that for a asymptotically positive infinite Galois tower $\rK = \{K_i\}$, either the minimal index or the number of low-lying zeros grows rapidly.

\begin{theorem}\label{Zeros}
Let $\rK= \{K_i\}_i$ be an asymptotically positive Galois extension of number fields. Then, under GRH, either
        $$
            I_{K_i} \gg e^{n_{K_i}^2} \hspace{5mm} \text{ or } \hspace{5mm} N_{K_i}(2) \gg n_{K_i}^2,
        $$
where the implied constants are absolute.
\end{theorem}

\bigskip

\section{\bf Lower bounds on discriminant}\label{section-3}
\medskip

The results stated in the earlier sections are obtained as a consequence of the more general and harder problem of finding lower bounds on the discriminant of number fields. Recall the following classical theorem of Mahler \cite{Mahler}.

\begin{theorem}[Mahler]\label{mahler}
    Let $f(x) = a_n x^n + \cdots + a_1 x + a_0 \in \C[x]$ be a polynomial with roots $\alpha_1, \alpha_2, \cdots, \alpha_n$. Let
    $$
        D:= a_n^{2n-2} \, \prod_{i>j} (\alpha_i - \alpha_j)^2
    $$
    be its discriminant. Then, 
    $$
        |D| \leq n^n M(f)^{2n-2},
    $$
    where $M(f)$ denotes the Mahler measure of $f$, given by
    $ M(f) := |a_n| \prod_{i} \max(1,|\alpha_i|)$.
\end{theorem}

For an algebraic integer $\alpha$, if  $K=\Q(\alpha)$, applying Theorem \ref{mahler} along with \eqref{index-disc}, we obtain
$$
2 h(\alpha) \geq \frac{\log d_K}{n_K^2} - \frac{\log n_K}{n_K} + \frac{2 \log I_K}{n_K^2} .
$$
Thus, in a family of number fields $\{K_i = \Q(\alpha_i)\}$ with $\alpha_i$ algebraic integers, if $\log  d_{K_i} \gg n_{K_i}^2$, then Bogomolov property holds for $\bigcup_i \alpha_i$. By Minkowski's bound, we know that $\log d_K \gg n_{K}$ and there are infinitely many families of number fields (Hilbert class field towers for instance), where $\log d_{K_i} \ll n_{K_i}$. In such cases, Mahler's bound does not yield any lower bound on the height of $\alpha$. With this motivation, we will prove the following lower bound on $\log d_K$ in terms of low-lying zeros and the splitting nature of primes.
   \begin{theorem}\label{Northcott}
    Let $K$ be a number field. Under GRH,
    $$
   \frac{\log d_{K}}{n_{K}} \, \geq \, 2 \, + \, \frac{1.168}{n_{K}} \,  N_{K}(1) \, + \, 2.032 \, \mathlarger{\mathlarger{\sum}}_{q} \, \frac{\mathcal{N}_{q}(K)}{n_{K}} \, \frac{\log q}{\sqrt{q}} \, e^{- 0.212(\log q)^2},
    $$
    where $q$ runs over all prime powers.
\end{theorem}

We note that the proof of the above theorem allows us to replace the term $\frac{1.168}{n_K} \, N_K(1)$  with
$$
    \frac{2.206\, \sqrt{\pi}}{n_K} \, \mathlarger{\mathlarger{\sum}}_{\substack{|t|\leq 2\\ \zeta_K(1/2 + it) =0}} \left( e^{-\frac{t^2}{0.848}}\, - \, e^{-\frac{1}{0.212}}\right)
$$
to obtain a better lower bound. For a non-zero algebraic integer $\alpha$ and $K=\Q(\alpha)$, using Mahler's bound in the above theorem, we get
$$
    2h(\alpha)  \, n_K \geq 2 + \frac{1.168}{n_{K}} \, N_{K}(1) + 2 \, \mathlarger{\mathlarger{\sum}}_{q} \, \frac{\mathcal{N}_{q}(K)}{n_{K}} \, \frac{\log q}{\sqrt{q}} \, e^{- 0.212(\log q)^2} - \frac{\log n_K}{n_K} + \frac{2 \log I_K}{n_K^2}.
$$
Hence, we deduce the following corollary.
\begin{corollary}
    Let $S$ be the set of all algebraic integers $\alpha$ and $K=\Q(\alpha)$ satisfying
$$
 1.168 \, N_K(1) +  2 \, \mathlarger{\mathlarger{\sum}}_{q}  \, \mathcal{N}_q(K) \, \frac{\log q}{\sqrt{q}} \, e^{- 0.212(\log q)^2} + \frac{2 \log I_K}{n_K} \geq \log n_K.
$$
Then, Lehmer's conjecture holds for $S$.
\end{corollary}

\medskip

Thus, for a family of algebraic integers $\{\alpha_i\}$, say $K_i = \Q(\alpha_i)$, for Lehmer's conjecture to hold for $\cup_i \alpha_i$, it is sufficient for either one of the following conditions to hold: $\zeta_{K_i}$ have densely concentrated low lying zeros, reasonably many prime ideals with small norms in $\{K_i\}$ or large values of minimal index $I_{K_i}$.

\medskip

Indeed, the lower bound for height in \eqref{B-Z} due to Bombieri-Zannier is a consequence of eventual complete splitting of primes in an infinite Galois extension. They also predicted that the sum in the RHS of \eqref{B-Z} is always finite. This was recently shown to be false by S. Checcoli and A. Fehm \cite{Checcoli}, who used a criterion due to M. Widmer \cite{Widmer} to construct an infinite extension where this sum diverges. For infinite non-Galois extensions, the analogue of this sum is
\begin{equation}\label{BZ-analog-question}
    \frac{1}{2} \mathlarger{\sum}_q \, \psi_q \, \frac{\log q}{q+1}.
\end{equation}
It is thus natural to ask when this sum converges and if it holds information regarding lower bounds on height. From Theorem \ref{mahler}, we know that a lower bound on ${\log d_K}$ yields a lower bound on height. In that regard, we will prove the following.

\begin{theorem}\label{Discriminant_Bound_2}
    Let $\rK = \{K_i\}_i$ be a tower of number fields. Then under GRH,
    $$
   \frac{\log d_{K_i}}{n_{K_i}} (1 + o(1)) \geq \gamma + \log 8\pi  + \frac{\sqrt{\pi \log n_{K_i}}}{n_{K_i}}  \mathlarger{\mathlarger{\sum}}_{|t| \leq 2}  \left( \frac{1}{n_{K_i}^{t^2/4}} - \frac{1}{n_{K_i}} \right) + 2 \mathlarger{\sum}_{q \leq \log n_{K_i}} \psi_q \frac{\log q}{\sqrt{q}} + o(1),
    $$
    where $t$ runs over the zeros of $\zeta_{K_i}(s)$ on $\Re(s)=1/2$ and $q$ runs over all the prime powers. The $o(1)$ term tends to $0$ as $i$ tends to infinity.
\end{theorem}

As a result, we obtain a trivial upper bound 
$$
\sum_{q \leq \log n_{K_i}} \psi_q  \,\, \frac{\log q}{\sqrt{q}} \,\, \leq \, \,  \frac{1}{2} \frac{\log d_{K_i}}{n_{K_i}} \, \left(1 + o(1)\right).
$$

\noindent

Thus, if $\log d_{K_i} \ll n_{K_i}$, then the sum in \eqref{BZ-analog-question} converges, which is in support of the Bombieri-Zannier prediction. This gives a wide range of infinite extensions where this sum is finite.

\medskip

\section{\bf Preliminaries}\label{section-4}
In this section, we state and develop the necessary ingredients towards the proof of our theorems. An important role is played by Weil's explicit formula. We use the version formulated by Guinand (see \cite[p. 122]{Odlyzko}).

\begin{theorem}[Guinand-Weil explicit formula]\label{explicit}
    Let $F(x)$ be a differentiable, even, positive function defined on the whole real line $\R$ such that $F(0)=1$ and there exist positive constants $c$ and $\epsilon$ such that
    $$
        F(x), F'(x) \leq c \, e^{-(1/2 + \epsilon)|x|}
    $$

as $|x| \to \infty$. Define
$$
    \phi(s) := \int_{-\infty}^{\infty} F(x) e^{(s-1/2)x}\, dx.
$$
Let $K/\Q$ be a number field. Denote by $n_K$ and $d_K$ the degree and absolute discriminant of $K$ over $\Q$. Let $r_1$ and $r_2$ denote the number of real and complex embedding of $K$. Then, we have
\begin{align*}
\log d_K = r_1 \frac{\pi}{2} \, + \, & n_K (\gamma + \log 8\pi) - n_K \int_0^{\infty} \frac{1-F(x)}{2\sinh{x/2}} \, dx \\
& - r_1 \int_0^{\infty} \frac{1-F(x)}{2\cosh{x/2}}\, dx - 4\int_0^{\infty} F(x) \cosh{x/2} \, dx + \sum_{\rho}{'} \phi(\rho) \\
&+ 2\sum_{\rp} \sum_{m=1}^{\infty} N(\rp)^{-m/2} F(m\log N(\rp)) \log N(\rp),
\end{align*}
where in the first sum $\rho$ runs over all zeros of the Dedekind zeta function $\zeta_K(s)$ in the critical strip, where $\rho$ and $\overline{\rho}$ are grouped together, $\rp$ runs over the prime ideals of $K$ and $N(\rp)$ denotes the norm of $\rp$.

\end{theorem}
\medskip

Applying the above explicit formula, we obtain the following lemmata.

\begin{lemma}\label{lemma 1}
    Let $K$ be a number field with  absolute discriminant $d_K$ and degree $n_K$ respectively. Then under GRH, for $y>0$
\begin{align*}
 \frac{\log d_K}{n_K} = \frac{\pi r_1}{2n_K}  + & (\gamma + \log 8\pi) - \frac{2 \sqrt{\pi}}{n_K}  \frac{e^{1/16y}}{\sqrt{y}} + \frac{1}{n_K} \sqrt{\frac{\pi}{y}} \, {\mathlarger{\mathlarger{\sum}}_{t}} e^{- \frac{t^2}{4y}} \nonumber\\
&+ 2 \,\mathlarger{\mathlarger{\sum}}_{q} \frac{\mathcal{N}_q(K)}{n_K} \sum_{m=1}^{\infty} q^{-m/2} e^{-y {(m \log q)}^2} \log q \, + \, O(y), 
\end{align*}
where  $t$ runs over the imaginary part of the non-trivial zeros of $\zeta_K(s)$ and $q$ runs over all prime powers.
\end{lemma}

\begin{proof}
    
For $y>0$, suppose
$$
    F(x) = e^{-yx^2}.
$$ 
Applying Theorem \ref{explicit}, under GRH, we obtain

\begin{align*}
\log d_K = r_1 \frac{\pi}{2} \,+ \, & n_K (\gamma + \log 8\pi) - n_K \int_0^{\infty} \frac{1-e^{-yx^2}}{2\sinh{x/2}} \, dx - r_1 \int_0^{\infty} \frac{1- e^{-yx^2}}{2\cosh{x/2}}\, dx\nonumber\\
&  - 4\int_0^{\infty} e^{-yx^2} \cosh{x/2} \, dx + \text{Re}\mathlarger{\mathlarger{\sum}}_t \int_{-\infty}^{\infty} e^{itx-yx^2}\, dx  \\
&+ 2\sum_{\rp} \sum_{m=1}^{\infty} N(\rp)^{-m/2} e^{-y m^2\log^2 N(\rp)} \log N(\rp),\nonumber
\end{align*}
where the first sum runs over all real $t$ such that $\zeta_K(1/2 + it) =0$. For our purposes, we need to divide both sides by $n_K$ and re-write the last sum above as follows.

\begin{align}\label{explicit_K}
\frac{\log d_K}{n_K} = \frac{\pi r_1}{2n_K}  \,+ \, &  (\gamma + \log 8\pi) -  \int_0^{\infty} \frac{1-e^{-yx^2}}{2\sinh{x/2}} \, dx 
 - \frac{r_1}{n_K} \int_0^{\infty} \frac{1- e^{-yx^2}}{2\cosh{x/2}}\, dx \nonumber\\
 &- \frac{4}{n_K}\int_0^{\infty} e^{-yx^2} \cosh{x/2} \, dx + \frac{1}{n_K}\, \text{Re}\mathlarger{\mathlarger{\sum}}_t \int_{-\infty}^{\infty} e^{itx-yx^2}\, dx  \nonumber\\
 &+ 2 \, \mathlarger{\mathlarger{\sum}}_{q} \frac{\mathcal{N}_q(K)}{n_K} \sum_{m=1}^{\infty} q^{-m/2} e^{-ym^2 \log^2 q}\, \log q ,
\end{align}
where in the last sum $q$ runs over all prime powers and $\mathcal{N}_q(K)$ is the number of prime ideals in $K$ with norm $q$. Note that
\begin{equation}\label{first int}
    0\leq \int_0^{\infty} \frac{1- e^{-yx^2}}{2\cosh{x/2}}\, dx \leq \int_0^{\infty} \frac{1-e^{-yx^2}}{2\sinh{x/2}} \, dx = O(y)
\end{equation}
and
\begin{equation}\label{sec int}
\frac{4}{n_K} \int_0^{\infty} e^{-yx^2} \cosh{x/2}\, dx = \frac{2e^{\frac{1}{16 y}}}{n_K} \sqrt{\frac{\pi}{y}} .
\end{equation}
\\
Using \eqref{first int} and \eqref{sec int}, the identity \eqref{explicit_K} can be written as
\begin{align}\label{explicit_K-modified}
\frac{\log d_K}{n_K} = \frac{\pi r_1}{2n_K}  \,+ \, &  (\gamma + \log 8\pi) + O(y) + \frac{1}{n_K}\, \text{Re}\mathlarger{\mathlarger{\sum}}_t \int_{-\infty}^{\infty} e^{itx-yx^2}\, dx \nonumber\\ 
 &+ 2 \, \mathlarger{\mathlarger{\sum}}_{q} \frac{\mathcal{N}_q(K)}{n_K} \sum_{m=1}^{\infty} q^{-m/2} e^{-ym^2 \log^2 q}\, \log q 
 - \frac{2e^{\frac{1}{16 y}}}{n_K} \sqrt{\frac{\pi}{y}} .
\end{align}
The integral inside the first sum is
\begin{equation*}
\int_{-\infty}^{\infty} e^{itx-yx^2}\, dx = e^{-t^2/4y} \int_{-\infty}^{\infty} e^{-y(x-\frac{it}{2y})^2}\, dx = \sqrt{\frac{\pi}{y}} e^{-t^2/4y} > 0.
\end{equation*}
Putting this in \eqref{explicit_K-modified}, we obtain the lemma.
\end{proof}
\medskip

\begin{lemma} \label{LEMMA 1}
Let $K$ be a number field of degree $n_K $ and absolute discriminant $d_K$. Then under GRH

$$
\log d_K = \bigg( 2-\frac{\pi}{2} \bigg) \, r_1 + (\gamma + \log 8\pi - 2) \, n_K - \frac{16}{3} + \mathlarger{\mathlarger{\sum}}_t \frac{2}{1+t^2} + 2 \, \mathlarger{\mathlarger{\sum}}_{q}  N_K(q) \frac{\log q}{q^{3/2}},
$$
  where $t$ runs over non-trivial zeros of $\zeta_K(s)$ and $q$ over all prime powers.  
\end{lemma}

\begin{proof}
\noindent
Applying the explicit formula Theorem \ref{explicit} with  $F(x)=e^{-|x|}$, we obtain
\begin{align*}
\log d_K = r_1 \frac{\pi}{2} + & n_K (\gamma + \log 8\pi) - n_K \int_0^{\infty} \frac{1-e^{-|x|}}{2\sinh{x/2}} \, dx - r_1 \int_0^{\infty} \frac{1- e^{-|x|}}{2\cosh{x/2}}\, dx\\
&  - 4\int_0^{\infty} e^{-|x|} \cosh{x/2} \, dx + \text{Re}\sum_t  \int_{-\infty}^{\infty} e^{itx-|x|}\, dx  \\
&+ 2\sum_{\rp} \sum_{m=1}^{\infty} N(\rp)^{-m/2} e^{- |m \log N(\rp)|} \log N(\rp).  
\end{align*}
Evaluating the integrals give
\begin{align*}
\log d_K = r_1 \frac{\pi}{2} + &  n_K (\gamma + \log 8\pi) - 2 n_K -  (\pi-2) r_1 - \frac{16}{3} + \mathlarger{\mathlarger{\sum}}_t \frac{2}{1+t^2} \\
& + 2\sum_{q} N_K(q) \sum_{m=1}^{\infty} q^{-m/2} e^{- m \log q} \log q \\
&= \left(2-\frac{\pi}{2}\right) r_1 + (\gamma + \log 8\pi - 2) n_K - \frac{16}{3} + \mathlarger{\mathlarger{\sum}}_t \frac{2}{1+t^2} + 2 \mathlarger{\mathlarger{\sum}}_{q} N_K(q) \frac{\log q}{q^{\frac{3}{2}}}.
\end{align*}
as required.
\end{proof}

\medskip

We also need an estimate on the number of zeros of the Dedekind zeta-function. For this purpose, we use the most recent result due to E. Hasanalizade, Q. Shen and P. J. Wong \cite{Hasanalizade}.

\begin{theorem}[Hasanalizade, Shen, Wong]\label{precise zero}
     Let $K/\Q$ be a number field and $N_K(T)$ be the number of zeroes $\rho$ of $\zeta_K(s)$ in the critical strip with $|Im(\rho)|\leq T$. Then,

     \begin{align*}
        \left|  N_K(T) - \frac{T}{\pi} \log \frac{d_K T^{n_K}}{(2 \pi e)^{n_K} } \right| \, \leq \, 0.228 \, \left( \log d_K + n_K \log T \right) + 23.108 \, n_K + 4.520
     \end{align*}
     for $T \geq 1$.
\end{theorem}

\medskip

\noindent
Now, we can calculate the contribution from the term on the RHS of Lemma \ref{lemma 1} arising from zeros of $\zeta_K(s)$.
\begin{lemma}\label{lemma 2} For $y > 0$, under GRH
\begin{align*}
 \frac{1}{n_K} \sqrt{\frac{\pi}{y}} \,\,{\mathlarger{\mathlarger{\sum}}_{t}} e^{- \frac{t^2}{4y}}   \, \geq \, \frac{1}{n_K} \sqrt{\frac{\pi}{y}} & \,\, {\mathlarger{\mathlarger{\sum}}_{|t|\leq 2 }} \left( e^{- \frac{t^2}{4y}} - e^{- \frac{1}{y}} \right) \, - \, 4.520 \frac{e^{-\frac{1}{y}}}{ n_K} \sqrt{\frac{\pi}{y}}\,   +  \, g(y) \nonumber \\
 & + \frac{\log d_K}{n_K} \left(\frac{2\, e^{-\frac{1}{y}}}{\sqrt{\pi y}}\, + G\left(\frac{1}{\sqrt{y}}\right) - 0.228 \, e^{-\frac{1}{y}} \, \sqrt{\frac{\pi}{y}}  \right) ,
 \end{align*}
 where $t$ runs over the imaginary part of the non-trivial zeros of $\zeta_K(s)$. Here $G$ denotes the Gaussian integral 
 $$
    G(y) = \frac{2}{\sqrt{\pi}} \int_y^{\infty} e^{-t^2} dt
$$ 
and $g$ is a function satisfying $g(y)$ tends to $ 0$ as $y$ tends to $ 0$.
\end{lemma}

\begin{proof}

For $T > 2$, by symmetry of zeros about the $x$-axis,
\begin{align*}
{\sum_{|t|\leq T }} e^{- \frac{t^2}{4y}} &= {\sum_{|t|\leq 2 }} e^{- \frac{t^2}{4y}} +  2 {\sum_{2< t \leq T }} e^{- \frac{t^2}{4y}} .
\end{align*}
\noindent
By partial summation, as $T\to \infty$
\begin{align}\label{Partial_Sum}
{\sum_{t}} e^{- \frac{t^2}{4y}} = {\sum_{|t|\leq 2 }} \Big( e^{- \frac{t^2}{4y}} - e^{- \frac{1}{y}} \Big) +  \frac{1}{2y} \int_{2}^{\infty} t \, N_K(t) \, e^{-\frac{t^2}{4y}} dt. 
\end{align}
\noindent
Denote by $
    E_K(y) := \frac{1}{2y} \int_{2}^{\infty} t N_K(t) e^{-\frac{t^2}{4y}} dt$. By Theorem \ref{precise zero}, we have
\begin{align*}
 E_K(y) &\geq \frac{1}{2y} \int_{2}^{\infty} t\,\, \left( \frac{t}{\pi} \log \frac{d_Kt^{n_K}}{(2 \pi e)^{n_K}} - 0.228 \log d_K t^{n_K} - 23.108 n_K - 4.520 \right) e^{-\frac{t^2}{4y}} dt.
\end{align*}

\noindent
Multiplying by $\frac{1}{n_K}\sqrt{\frac{\pi}{y}}$, a routine computation yields
\begin{align}\label{Partial_Sum3}
\frac{1}{n_K}\sqrt{\frac{\pi}{y}} \,\,E_K(y)
\, \geq \, \frac{\log d_K}{n_K} & \left(\frac{2 \,  e^{-\frac{1}{y}}}{\sqrt{\pi y}} + G\left(\frac{1}{\sqrt{y}}\right) - 0.228 \sqrt{\frac{\pi}{y}} e^{-\frac{1}{y}} \right) \nonumber\\
&- 4.520 \frac{1}{ n_K} \sqrt{\frac{\pi}{y}} e^{-\frac{1}{y}} + g(y),
\end{align}
where 
\begin{align*}
 g(y) = \frac{1}{2 \sqrt{\pi}\, y^{3/2}} & \Bigg( \int_{2}^{\infty} t^2 \, e^{-\frac{t^2}{4y}}  \,(\log t) \,  dt - \log 2\pi e \int_{2}^{\infty} \, t^2 \, e^{-\frac{t^2}{4y}}\,   dt \\
 &-0.228 \pi \int_{2}^{\infty}\, t\, e^{-\frac{t^2}{4y}}\, (\log t) \,  dt - 23.108 \pi \int_{2}^{\infty} \, t\, e^{-\frac{t^2}{4y}} \, dt \Bigg).
\end{align*}
\noindent
Combining (\ref{Partial_Sum}) and (\ref{Partial_Sum3}), we obtain the lemma.
\end{proof}
\medskip

We now show that the sum over the zeros of $\zeta_K(s)$ in Lemma \ref{LEMMA 1} can be estimated as follows.
\begin{lemma}\label{LEMMA 2}
Under GRH
\begin{align*}
  \mathlarger{\mathlarger{\sum}}_t \frac{2}{1+t^2} = 2 \mathlarger{\mathlarger{\sum}}_{|t| \leq 2} \Big( \frac{1}{1+t^2} - \frac{1}{5} \Big) + \big( 0.548 + c_1 \big) \log d_K  + (0.309 + c_2) n_K + O(1),
\end{align*}
where $t$ runs over the imaginary part of all the non-trivial zeros of $\zeta_K(s)$ and $|c_1| \leq 0.0912$ and $|c_2| \leq 9.3572$.
\end{lemma}
\begin{proof}
By Theorem {\ref{precise zero}}, the number of zeros of $\zeta_K(s)$ with height $ |t| \leq T$ is
\begin{align}\label{zero_last}
\left| N_K(T) - \frac{T}{\pi} \log \frac{ d_K T^{n_K}}{(2 \pi e)^{n_K}}\right | \leq 0.228 \, \log d_K + 0.228 \, n_K \log T + 23.108 \, n_K + 4.520.
\end{align}

\medskip

\noindent
Since the zeros of $\zeta_K(s)$ are symmetric about the real line, for $T >2$, 
\begin{align*}
    \sum_{|t| \leq T} \frac{2}{1+t^2} &=  \sum_{ |t| \leq 2} \frac{2}{1+t^2}  + 2 \sum_{ 2 < t \leq T} \frac{2}{1+t^2}.
\end{align*}
 
\noindent
As $ T\rightarrow \infty$, using partial summation we obtain
\begin{align}\label{Z}
    \sum_{t} \frac{2}{1+t^2} &=  \sum_{ |t| \leq 2}\frac{2}{1+t^2} - \frac{2}{5} N_K(2) + 4 \int_2^{\infty} \frac{t}{(1+t^2)^2}\, N_K(t) \, dt .
\end{align}
\medskip
\noindent
Using \eqref{zero_last}, a standard computation shows that the contribution of the main term to the above integral is
\begin{align*}
  & 4 \int_2^{\infty} \frac{t}{(1+t^2)^2} \left( \frac{t}{\pi} \log \frac{d_K t^{n_K}}{(2 \pi e)^{n_K}} \right) \,  dt \\
  &= \frac{4}{\pi} \Big( \log d_K -  n_K\log (2 \pi e)  \Big) \int_2^{\infty} \frac{t^2}{(1+t^2)^2}dt + \frac{4 n_K}{\pi} \int_2^{\infty} \frac{t^2 \log t}{(1+t^2)^2}dt \\[2mm]
   &= \frac{4}{\pi} \Big( \log d_K -  n_K \log (2 \pi e) \Big) (0.431) + \frac{4 n_K}{\pi} (0.774)\\[2mm]
   &= 0.548 \, \log d_K + 0.309 \, n_K.
\end{align*}

\noindent
Moreover, the contribution of the rest of the terms in \eqref{zero_last} to the integral in \eqref{Z} is at most
\begin{align*}
    4 (0.228 \log d_K   + 23.108 n_K + 4.520) & \int_2^{\infty} \frac{t}{(1+t^2)^2} \, dt + 0.228 n_K \int_2^{\infty} \frac{t \log t}{(1+t^2)^2}\, dt\\
    \vspace{2mm}
    & = 0.0912 \, \log d_K + 9.3572 \, n_K + 1.808.
\end{align*}
Putting this together, we obtain the lemma.
\end{proof}

\bigskip

Let $\pi(x)$ denote the number of primes less than $x$. Under Riemann hypothesis for $\zeta(s)$, the prime number theorem gives
$$
    \pi(x) = Li(x) + O\left( x^{1/2} \log x\right),
$$
where $Li(x) = \int_2^x \frac{1}{\log t}\, dt$. For a non-negative, smooth and real valued function $f(x)$, employing Abel's summation formula, it is easy to obtain the following lemma.

\begin{lemma}\label{Sum}
    Let $R(x)= \pi(x)-Li(x)$. If $f(x)$ is a non-negative, smooth and real valued function satisfying
    $$
  Li(2)f(2) + f(x) R(x) - \int_2^x R(x) f^{\prime}(t) dt = o\left(\int_2^x \frac{f(t)}{\log t} dt\right),
    $$
    then as $x $ tends to infinity
    $$
        \sum_{p \leq x} f(p) \sim \int_2^x \frac{f(t)}{\log t} dt.
    $$
\end{lemma}
\medskip

Lastly, we state the following lemma due to E. Bombieri and U. Zannier (see \cite[Equation 10]{BZ}).

\begin{lemma}[Bombieri, Zannier]\label{BZ-intermediate}
    Let $K/\Q$ be a Galois extension and $\alpha \in K^*$. Let $f(x)$ be the minimal polynomial of $\alpha$ of degree $m$ and $D(f)$ denote the discriminant of $f$. For a rational prime $p$, let $\nu$ be a valuation on $K$ above $p$. Denote by $e_p$ and $f_p$ the ramification index and residue class index of $K_{\nu} / \Q_p$. Then,
    $$
        \log |D(f)| \geq m^2 \sum_{q< m} \frac{1}{e_p} \left( V_p(\alpha;K) + \frac{1}{q+1} - \frac{1}{m} \right) \log p,
    $$
    where $q$ runs over all prime powers and $V_p(\alpha;K)$ denotes the normalized variance given by
    $$
        V_p(\alpha; K) := \frac{1}{m^2} \sum_{x\in \mathbb{F}_q \cup \infty} \left(N_x - \frac{m}{q+1}\right)^2,
    $$
    where $N_x$ denotes the number of conjugates $\alpha_i$ of $\alpha$ with reduction $x \in \mathbb{F}_q$.
    
\end{lemma}

\bigskip

\section{\bf Proof of the main theorems}\label{section-5}
\medskip

\begin{proof}[Proof of Theorem \ref{Lehmer-GRH}]
From Lemma \ref{LEMMA 1} and Lemma \ref{LEMMA 2}, we have
\begin{align*}
    \log d_K &=  2 \sum_{|t| \leq 2} \left( \frac{1}{1+t^2} - \frac{1}{5} \right) + \left( 0.548 + c_1 \right) \log d_K  + O\left(n_K\right) \\
    &\geq 2 \sum_{|t| \leq 2} \left( \frac{1}{1+t^2} - \frac{1}{5} \right) + 0.456 \log d_K +  O\left(n_K \right).
 \end{align*}

\noindent
Therefore
\begin{align}\label{Discriminant}
    \log d_K \geq 3.67 \sum_{|t| \leq 2} \left( \frac{1}{1+t^2} - \frac{1}{5} \right) +  O\left( n_K \right).
\end{align}

\noindent
Let $K=\Q(\alpha) $ and $f(x)\in\Z[x]$ be the minimal polynomial of $\alpha$. From Theorem \ref{mahler}, 
\begin{align*}
 \log |D(f)| \leq n_K \log n_K + 2(n_K-1) n_K h(\alpha). \nonumber 
\end{align*}
Since Lehmer's conjecture holds for all algebraic non-integers, we assume that $\alpha$ is an algebraic integer. Therefore,
\vspace{0.1cm}
$$
|D(f)| = I(\alpha)^2 \, |d_K| \geq |d_K|,
$$

\noindent
where $I(\alpha)$ is the index of $\alpha$ which is always $\geq 1$. Hence, by (\ref{Discriminant}), 
\begin{align*}
     2(n_K-1) n_K h(\alpha) + n_K \log n_K \geq  \log d_K 
     &\geq 3.67 \sum_{|t| \leq 2} \Big( \frac{1}{1+t^2} - \frac{1}{5} \Big) + O\left(n_K\right).
\end{align*}

\noindent
 Dividing both sides by $n_K$, we obtain
\begin{align*}
    2 n_K h(\alpha) \geq \frac{3.67}{n_K} \,\left(\lambda_K(2)-\frac{1}{5}N_K(2)\right) - \log n_K + O(1),
\end{align*}
as required.
\end{proof}
\medskip

\begin{proof}[Proof of Theorem \ref{Uncondional}]
    For a fixed $\epsilon > 0$, let $\mathcal{S}$ be a set of primes $< Y$ with cardinality at least $\epsilon \, \pi(Y)$. Suppose $\alpha$ is an algebraic integer such that $K=\Q(\alpha)$ satisfies the condition $\frac{\mathcal{N}_p(K)}{ n_K }>\delta >0$ for $p \in \mathcal{S}$. Then, the arithmetic term in Lemma \ref{lemma 1} can be bounded below as follows.

\begin{align} \label{AA}
     2 \mathlarger{\mathlarger{\sum}}_{q} \frac{\mathcal{N}_q(K)}{n_{K}} \sum_{m=1}^{\infty}  q^{-m/2}  e^{-y (m \log q)^2}\log q & \, \geq  2 \delta \sum_{p \in \mathcal{S}}  p^{-1/2}  e^{-y (\log p)^2}\log p \nonumber\\
     & \,\geq \,  2 \delta \sum_{(1-\epsilon)Y < \,p  \,\leq Y}  p^{-1/2}  e^{-y (\log p)^2}\log p.
\end{align}
\medskip
\noindent
Choose $y= (\log n_K)^{-1}$ and define
$$
    f(x) = x^{-1/2} e^{-\frac{(\log x)^2}{\log n_{K}}} \log x.
$$
Now applying Lemma \ref{Sum}, under the Riemann hypothesis, $| R(x) | \leq O(x^{1/2} \log x)$ and hence for large values of $n_K$, we obtain
\begin{align}\label{CC}
    \mathlarger{\mathlarger{\sum}}_{X < p \leq Y }  p^{-1/2}  e^{-\frac{(\log p)^2}{\log n_{K}}}\log p &\gg   \int_X^Y \Big( x^{-1/2} e^{-\frac{(\log x)^2}{\log n_{K}}} \Big)  dx \nonumber \\
    &= \int_{\log X}^{\log Y}  e^{-\frac{t^2}{\log n_{K}}+ \frac{t}{2}}  dt \gg e^{-\frac{(\log Y)^2}{\log n_K}}  \sqrt{Y} . 
\end{align}

\noindent
Since
$$
    n_K^{1/2} \geq Y \geq (\log n_K)^{2},
$$
for large values of $n_K$,
$$
A(n_K) := e^{-\frac{(\log Y)^2}{\log n_K}}  \sqrt{Y} - \log n_K
$$
is positive and in fact tends to infinity as $n_K$ tends to infinity.
\noindent
Therefore, from equation \eqref{AA}

$$
2 \, \mathlarger{\mathlarger{\sum}}_{q} \frac{\mathcal{N}_q(K)}{n_{K}} q^{-1/2}  e^{-\frac{(\log q)^2}{\log n_{K}}} \log q \gg \delta \,\,\frac{1}{\log n_K}\,\, e^{\frac{\sqrt{ \log n_{K} \log \log n_{K}}}{2}}.
$$

\noindent
Now applying Lemma \ref{lemma 1}
\begin{align*}
\frac{\log d_{K}}{n_{K}} &\geq (\gamma + \log 8\pi) +2 \mathlarger{\mathlarger{\sum}}_{q} \frac{\mathcal{N}_q(K)}{n_{K}} q^{-1/2}  e^{-\frac{(\log q)^2}{\log n_{K}}} \log q + o(1)\\
& \geq (\gamma + \log 8\pi) + \log n_K + A(n_K).
\end{align*}
Since $A(n_K)\to \infty$ as $n_K \to \infty$, by Theorem \ref{mahler}
$$
    n_K  \, h(\alpha) \rightarrow \infty
$$
as $n_K$ tends to infinity. 

\end{proof}

\medskip

\begin{proof}[Proof of Theorem \ref{Zeros}]
Let $ \rK = \{K_i\}$ be an asymptotically positive tower of Galois number fields. Suppose $K_i = \Q( \alpha_i)$ for algebraic integers $\alpha_i$ satisfying $I(K_i) = I({\alpha}_i)$, where $I_{K_i}$ is the minimal index of $K_i$. Denote by $f_i$ the minimal polynomial of $\alpha_i$.\\ 

\noindent
Applying Lemma \ref{BZ-intermediate} and using the fact that $V_p(\alpha, K) \geq 0$, we obtain
\begin{align}\label{BZ}
 2 \log I({\alpha}_i) + \log d_{K_i}  = \log |D(f_i)| \geq {n_{K_i}}^2 \sum_{q < n_{K_i}} \frac{1}{{e_p}} \Big( \frac{1}{q+1} - \frac{1}{n_{K_i}}\Big) \log p, 
\end{align}
where $q= p^{f_p}$. Since $\rK$ is asymptotically positive, for sufficiently large $i$, we have 
$$
    2 \log I({\alpha}_i) + \log d_{K_i}  = \log |D(f_i)| \gg n_{K_i}^2.
$$
Hence, at least one of $\log I({\alpha}_i)$ or 
    $\log d_{K_i}$ is  $\gg n_{K_i}^2$. Now, from Lemma \ref{LEMMA 1} and \ref{LEMMA 2}, we have
\begin{align*}
\log d_K  \leq 2 N_K(2) + 0.629 \log d_K +  O\left(n_K\right),
\end{align*}
which implies
\begin{align}\label{BZ2}
    \log d_K \leq 5.4 N_K(2) + O(n_K).
\end{align}
Therefore, if $\log I(\alpha_i)$ is not $\gg n_{K_i}^2$, then
$$
    N_{K_i}(2) \gg {n_{K_i}}^2
$$
for sufficiently large $i$.
\end{proof}
\medskip

\begin{proof}[Proof of Theorem \ref{Northcott}]
From (\ref{explicit_K}) and (\ref{explicit_K-modified}) of Lemma \ref{lemma 1}, under GRH

\begin{align}\label{LB_1}
 \frac{\log d_K}{n_K} = \frac{\pi r_1}{2n_K}  \, + \, &  (\gamma + \log 8\pi) -  \int_0^{\infty} \frac{1-e^{-yx^2}}{2\sinh{x/2}} \, dx 
 - \frac{r_1}{n_K} \int_0^{\infty} \frac{1- e^{-yx^2}}{2\cosh{x/2}}\, dx \nonumber\\
 &- \frac{4}{n_K}\int_0^{\infty} e^{-yx^2} \cosh{x/2} \, dx + \frac{1}{n_K} \sqrt{\frac{\pi}{y}}{\sum_{t}} \,  e^{- \frac{t^2}{4y}} \nonumber\\
&+ 2 \mathlarger{\mathlarger{\sum}}_{q} \frac{\mathcal{N}_q(K)}{n_K} \sum_{m=1}^{\infty} q^{-m/2} e^{-y {(m \log q)}^2} \log q, 
\end{align}
where  $t$ runs over the imaginary part of the non-trivial zeros of $\zeta_K(s)$ and $q$ runs over all prime powers. Using Lemma \ref{lemma 2}, we have

\begin{align}\label{LB_2}
 \frac{1}{n_K} \sqrt{\frac{\pi}{y}} \,\,{\sum_{t}} e^{- \frac{t^2}{4y}}    & \geq \frac{1}{n_K} \sqrt{\frac{\pi}{y}} \,\, {\sum_{|t|\leq 2 }} \Big( e^{- \frac{t^2}{4y}} - e^{- \frac{1}{y}} \Big) - 4.520 \frac{1}{ n_K} \sqrt{\frac{\pi}{y}}\,\, e^{-\frac{1}{y}} + g(y) \nonumber \\
 & + \frac{\log d_K}{n_K} \left(\frac{2}{\sqrt{\pi y}}\,\, e^{-\frac{1}{y}} + G\left(\frac{1}{\sqrt{y}}\right) - 0.228 \sqrt{\frac{\pi}{y}} \,\,e^{-\frac{1}{y}} \right).
 \end{align}

 \medskip

 \noindent
 Define 
 \begin{align*}
 & H(y) := g(y) - \int_0^{\infty} \frac{1-e^{-yx^2}}{2\sinh{x/2}} dx,\\
 & F_1(y) := \frac{2}{\sqrt{\pi y}} e^{-\frac{1}{y}}+ G(\frac{1}{\sqrt{y}}) - 0.228 \sqrt{\frac{\pi}{y}} e^{-\frac{1}{y}},\\
& F_2(y) := 4.520 \frac{1}{ n_K} \sqrt{\frac{\pi}{y}} e^{-\frac{1}{y}} +  4 \int_0^{\infty} e^{-yx^2} \cosh{x/2} \, dx.
 \end{align*}
 Using the fact that 
 $$\left(\frac{\pi}{2}-\int_0^{\infty} \frac{1- e^{-yx^2}}{2\cosh{x/2}} dx \right) \geq 0
 $$
 for any $y > 0$, we obtain 

\begin{align}\label{LB_3}
    \frac{\log d_K}{n_K} \,\, &\geq \,\, \frac{ H(y) + (\gamma + \log 8\pi)}{(1-F_1(y))} -  \frac{1}{n_K} \frac{F_2(y)}{(1-F_1(y))} + \frac{1}{(1-F_1(y))}{\sum_{|t|\leq 2 }} \Big( e^{- \frac{t^2}{4y}} - e^{- \frac{1}{y}} \Big) \nonumber \\
    &+ \frac{2}{(1-F_1(y))} \mathlarger{\mathlarger{\sum}}_{q} \frac{\mathcal{N}_q(K)}{n_K}\sum_{m=1}^{\infty} q^{-m/2} e^{-y (m \log q)^2} \log q.
\end{align}

\medskip

\noindent
The function $\frac{H(y) + (\gamma + \log 8\pi)}{1 - F_1(y)} $ tends to its maximum value $\gamma + \log 8\pi$ as $y$ approaches $0$. Now substituting $y = 0.212$ in equation (\ref{LB_3}), 

\begin{align*} 
    \frac{\log d_K}{n_K} \,\, \geq\,\,  2 \, + \, \frac{1.016}{n_K} \, \sqrt{\frac{\pi}{0.212}} \, &{\mathlarger{\mathlarger{\sum}}_{|t|\leq 2 }} \left( e^{- \frac{t^2}{0.848}} - e^{- \frac{1}{0.212}} \right)  + \frac{1}{n_K} O(1)   \nonumber \\
    & \,\, + 2.032 \,  \mathlarger{\mathlarger{\sum}}_{q} \,\frac{\mathcal{N}_q(K)}{n_K} \sum_{m=1}^{\infty} q^{-m/2} \,  e^{-0.212 (m \log q)^2} \log q.
\end{align*}

\medskip

\noindent
Since
$$
\frac{1.016}{n_K} \sqrt{\frac{\pi}{ 0.212}} \, \, {\mathlarger{\mathlarger{\sum}}_{|t|\leq 2 }} \left( e^{- \frac{t^2}{  0.848}} - e^{- \frac{1}{ 0.212}} \right) \geq \frac{1.168}{n_K} N_K(1),
$$
\\
\noindent
we obtain
$$
\frac{\log d_{K}}{n_{K}} \, \geq \, 2 \, + \, \frac{1.168}{n_{K}} \,  N_{K}(1) \, + \, 2.032 \, \mathlarger{\mathlarger{\sum}}_{q} \, \frac{\mathcal{N}_{q}(K)}{n_{K}} \, \frac{\log q}{\sqrt{q}} \, e^{- 0.212(\log q)^2}.
$$
\noindent
 This concludes the proof of our theorem.
\end{proof}
\medskip

\begin{proof}[Proof of Theorem \ref{Discriminant_Bound_2}]
    Let $\rK$ be an infinite extension of $\Q$. We first show that for any $\epsilon > 0$, a number field $K \subset \rK$ with sufficiently large $n_K$ satisfies,
\begin{equation}\label{eqn4.4}
   \mathlarger{\mathlarger{\sum}}_{q} \frac{\mathcal{N}_q(K)}{n_K} \sum_{m=1}^{\infty} q^{-m/2} e^{- \frac{(m \log q)^2}{\log n_K}} \log q   \,\,\, \geq \, (1-\epsilon) \,\,\,  \mathlarger{\mathlarger{\sum}}_{q \leq \log n_K} \frac{\psi_q}{\sqrt{q}} \log q.
\end{equation}
Using the fact that $e^{-x} > 1-x$ for all $x>0$,
\begin{align*}
    \mathlarger{\mathlarger{\sum}}_{q} \frac{\mathcal{N}_q(K)}{n_K} \sum_{m=1}^{\infty} q^{-m/2} e^{- \frac{(m \log q)^2}{\log n_K}} \log q 
        \,\,\, & \geq \,\,\, \mathlarger{\mathlarger{\sum}}_{q \leq \log n_K} \frac{\mathcal{N}_q(K)}{n_K}  \frac{\log q}{\sqrt{q}} \left( 1-  \frac{\log^2 q}{\log n_K} \right)\\
        & \geq \,\,\,  \mathlarger{\mathlarger{\sum}}_{q \leq \log n_K} \frac{\mathcal{N}_q(K)}{n_K}  \frac{\log q}{\sqrt{q}} \left( 1 + O\left( \frac{(\log \log n_K)^2}{\log n_K} \right) \right).
\end{align*}
Thus, for any $\epsilon > 0$, for sufficiently large value of $n_K$, by Lemma \ref{lemma-1} 
    \begin{align*}
         2 \mathlarger{\mathlarger{\sum}}_{q} \frac{\mathcal{N}_q(K)}{n_K} \sum_{m=1}^{\infty} q^{-m/2} e^{-y (m \log q)^2} \log q  \,\,\, &\geq \,\,\, 2(1-\epsilon) \mathlarger{\mathlarger{\sum}}_{q \leq \log n_K} \frac{\mathcal{N}_q(K)}{n_K}  \frac{\log q}{\sqrt{q}} \\
          \,\,\, &\geq \,\,\, 2(1-\epsilon) \mathlarger{\mathlarger{\sum}}_{q \leq \log n_K} \psi_q  \frac{\log q}{\sqrt{q}},
    \end{align*}
which proves \eqref{eqn4.4}. Now, substituting $y= \frac{1}{\log n_K}$ in Lemma \ref{lemma 2}, we get
\begin{align*}
 \frac{1}{n_K} \sqrt{\frac{\pi}{y}} \, {\mathlarger{\mathlarger{\sum}}_{t}} e^{- \frac{t^2 \log n_K}{4}}    
 &\geq  \frac{\sqrt{\pi \log n_K}}{n_K} \,{\mathlarger{\mathlarger{\sum}}_{|t|\leq 2 }} \left( \frac{1}{{(n_K)}^{t^2/4}} - \frac{1}{n_K} \right) + \frac{\log d_K}{n_K} o(1) + o(1).
 \end{align*}
Putting this in Lemma \ref{LEMMA 1} gives
\begin{align*}\label{explicit_main_last}
 \frac{\log d_K}{n_K} \geq (\gamma + \log 8\pi) &+ \frac{\sqrt{\pi \log n_K}}{n_K} \,{\mathlarger{\mathlarger{\sum}}_{|t|\leq 2 }} \left( \frac{1}{{(n_K)}^{t^2/4}} - \frac{1}{n_K} \right) + \frac{\log d_K}{n_K} o(1) + o(1)\\
&+ 2 \mathlarger{\mathlarger{\sum}}_{q} \frac{\mathcal{N}_q(K)}{n_K} \sum_{m=1}^{\infty} q^{-m/2} e^{- \frac{{(m \log q)}^2}{\log n_K}} \log q.
\end{align*}
Finally, using \eqref{eqn4.4}, we deduce 
\begin{align*}
    \frac{\log d_K}{n_K} \Big( 1 + o(1) \Big) \geq (\gamma + \log 8\pi) + \frac{\sqrt{\pi \log n_K}}{n_K} \,{\mathlarger{\mathlarger{\sum}}_{|t|\leq 2 }} \left( \frac{1}{{n_K}^{t^2/4}} - \frac{1}{n_K} \right) + \, (1-\epsilon) \mathlarger{\mathlarger{\sum}}_{q \leq \log n_K} \psi_q \frac{\log q}{\sqrt{q}} + o(1), \nonumber
\end{align*}
which completes our proof.

\end{proof}

\bigskip

\bigskip

\section{\bf Concluding remarks}
\bigskip

\subsection{\bf Some Computational data}
\medskip

Let $\alpha$ be a non-zero algebraic integer and $K:=\Q(\alpha)$. As noted earlier, from Mahler's bound (see Theorem \ref{mahler}), one can obtain a lower bound for $h(\alpha)$ if $\log d_K \geq n_K^2$. Therefore it is important to provide examples in support of Theorem \ref{Lehmer-GRH}, where $\log d_K < n_K^2$.\\ 

We also note that in a family of number fields $K_i= \Q(\alpha_i)$ with $\alpha_i$ algebraic integers, if $N_{K_i}(2) > n_{K_i}\log n_{K_i}$, then using Mahler's bound (Theorem \ref{mahler}) and Theorem \ref{precise zero}, it is easy to show that Lehmer's conjecture holds for $\cup \alpha_i$. But Theorem \ref{Lehmer-GRH} is stronger than this as it includes cases where $N_K(2) < n_K\, \log n_K$. We demonstrate this with computational data below.\\

\begin{table}[h]
\caption{Computational data in support of Theorem \ref{Lehmer-GRH}}
    \centering
   \bgroup
\def\arraystretch{1.75}%
\begin{tabular}{|c|c|c|c|c|}
\hline
\rule{0pt}{1.5em} $m_{\alpha}(x)$ & $\log(d_K)$ &  $3.67 \left( \lambda_K(2) - \frac{1}{5} N_K(2) \right)$ & $n_K \log(n_K)$ & $N_K(2)$ \\ [1ex]
\hline

$x^3 + 18x^2 + 312$ & $8.05801080080209$ & $3.42934404079907$ & $3.29583686600433$ & $2$ \\ \hline

$x^3 + 5x^2 + 235$ & $7.06902342657826$ & $3.40554888853991$ & $3.29583686600433$ & $2$ \\ \hline

$x^3 + 3x + 213$ & $8.35208267135264$ & $3.71716791990380$ & $3.29583686600433$ & $4$ \\ \hline

$x^3 + 3x + 2613$ & $8.91985437219167$ & $4.84445187879911$ & $3.29583686600433$ & $4$ \\ \hline

$x^4 + 3x^2 + 30$ & $15.5928465065266$ & $5.98680373865722$ & $5.54517744447956$ & $6$ \\ \hline

$x^4 + 3x^2 + 1650$ & $14.2893667565255$ & $6.16211623755126$ & $5.54517744447956$ & $4$ \\ \hline

$x^4 + 3x^2 + 2109$ & $12.6237824800548$ & $6.33826295082401$ & $5.54517744447956$ & $4$ \\ \hline

$x^4 + 18x^2 + 60$ & $12.9559781599087$ & $6.48197134982413$ & $5.54517744447956$ & $4$ \\ \hline

 $x^5 + 42$ & $22.9978680353040$ & $10.3144599678732$ & $8.04718956217050$ & $8$ \\ \hline
 
 $x^5+2x^2+26$ & $21.0796386344435$ & $8.72232900418632$ & $8.04718956217050$ & $8$ \\ \hline
 
$x^6 + 65$ & $31.6224931648465$ & $11.4961494891968$ & $10.7505568153683$ & $8$ \\ \hline

$x^6 + 85$ & $32.9638130978199$ & $16.3097029958646$ & $10.7505568153683$ & $12$ \\
\hline
\end{tabular}
\egroup
\\

\end{table}


\medskip

\subsection{\bf Euler-Kronecker constants}\label{euler-kronecker}
\medskip

We now return to the question of Bombieri-Zannier. More precisely, for an infinite tower $\rK=\{K_i\}$ over $\Q$, we are interested the condition under which
$$
    \sum_q \psi_q \frac{\log q}{q}
$$
is bounded. Here, we highlight the connection of this question to bounds on Euler-Kronecker constants.\\

The Euler-Mascheroni constant denoted by $\gamma$ is defined as
\begin{equation*}
    \gamma := \lim_{x\to\infty} \left( \sum_{n\leq x} \frac{1}{n} - \log x\right). 
\end{equation*}
It can also be described as the constant term in the Laurent expansion of the Riemann zeta-function,
\begin{equation}\label{zeta-gamma}
    \zeta(s) = \frac{1}{s-1} + \gamma + O(s-1).
\end{equation}

A generalization of $\gamma$ to any number field $K$ was introduced by Y. Ihara \cite{Ihara1} using the Dedekind zeta-function $\zeta_K(s)$. Since $\zeta_K(s)$ has a simple pole at $s=1$, its Laurent expansion can be written in the form
\begin{equation*}
    \zeta_K(s) = \frac{c_{-1}}{s-1} + c_0 + O(s-1).
\end{equation*}
The Euler-Kronecker constant associated to $K$ is defined as
\begin{equation*}
    \gamma_K := \frac{c_{0}}{c_{-1}}.
\end{equation*}
One could also view $\gamma_K$ as the constant term in the Laurent expansion of the logarithmic derivative of $\zeta_K(s)$ at $s=1$, i.e.,
\begin{equation}\label{log-derivative-gamma_K}
  -  \frac{\zeta_K'}{\zeta_K} (s) = \frac{1}{s-1} -\gamma_K + O(s-1).
\end{equation}

Recall a lemma of H. M. Stark \cite{stark} derived from the Hadamard factorization theorem applied to the Dedekind zeta-function.
\medskip

\begin{lemma}[Stark] \label{stark-lemma} Let $K$ be an algebraic number field of degree $n_K=r_1+2r_2$, where $K$ has $r_1$ real
conjugate fields and $2r_2$ complex conjugate fields. Recall that 
$$
\psi(s) := \frac{\Gamma'(s)}{\Gamma(s)}
$$ 
is the digamma function. Then, for any $s \in {\mathbb C}, $
\begin{equation}\label{basic}
- {\zeta_K'(s)  \over \zeta_K(s)} - {1\over s-1} + \sum_{\rho} {1\over s-\rho} = {1\over 2}\log d_K
+ \left( {1\over s} - {n_K\over 2}\log \pi\right) + {r_1\over 2} \psi\left({s \over 2}\right) + r_2 \left( \psi(s) -\log 2\right) ,
\end{equation}
where the summation is over the non-trivial zeros of $\zeta_K(s)$.
\end{lemma}

\noindent
Taking $s\to 1^+$ in \eqref{basic}, using \eqref{log-derivative-gamma_K} and dividing by $n_K$, we obtain
\begin{equation}\label{gamma-zeros}
    -\frac{\gamma_K}{n_K} =  \frac{\log d_K}{2n_K} - \frac{1}{n_K}\, \sum_{\rho}\frac{1}{\rho} + O(1),
\end{equation}
where $\rho$ runs over all the non-trivial zeros of $\zeta_K(s)$ and the error term is independent of $K$.\\

\medskip
\noindent
For $\Re(s)>1$, we can write
\begin{equation*}
-\frac{\zeta_K'}{\zeta_K}(s) = \mathlarger{\mathlarger{\sum}}_q \mathcal{N}_q(K) \sum_{m=1}^{\infty} \frac{\log q}{q^{ms}},
\end{equation*}
where $q$ runs over all prime powers. Dividing by $n_K$, for $s = 1+\sigma > 1$,
\begin{equation}\label{log-derivative}
    -\frac{1}{n_K}\,\frac{\zeta_K'}{\zeta_K}(1+\sigma) = \mathlarger{\mathlarger{\sum}}_q \frac{\mathcal{N}_q(K)}{n_K} \sum_{m=1}^{\infty} \frac{\log q}{q^{ms}} = \mathlarger{\mathlarger{\sum}}_q \frac{\mathcal{N}_q(K)}{n_K} \, \frac{\log q}{q^{\sigma+1}-1}.
\end{equation}
Furthermore, by \eqref{basic} for $\sigma>0$
\begin{align}\label{basic-1}
    - \frac{1}{n_K} \, {\zeta_K' \over \zeta_K} (1+\sigma) &= {1\over \sigma \, n_K} - \sum_{\rho} {1\over \rho+\sigma} + {\log |d_K| \over 2 n_K}
+ \left( {1\over (1+\sigma)\, n_K} - {\log \pi\over 2}\right) \nonumber\\ 
    &+ {r_1\over 2n_K} \psi\left({1+\sigma \over 2}\right) + \frac{r_2}{n_K} \left( \psi(1+\sigma) -\log 2\right).
\end{align}

\noindent
Combining \eqref{log-derivative} and \eqref{basic-1} and putting $\sigma =  1/\sqrt{n_{K}}$, we obtain
\begin{equation}\label{euler-2}
    \mathlarger{\mathlarger{\sum}}_q \frac{\mathcal{N}_q(K)}{n_K} \, \frac{\log q}{q^{\sigma+1}-1} = \frac{\log d_K}{2n_K} - \frac{1}{n_K}\, \sum_{\rho}\frac{1}{\rho + \sigma} + O(1).
\end{equation}

\noindent
For an infinite tower $\rK=\{K_i\}$ over $\Q$ and a fixed $\sigma > 0$,
$$
    \sum_q \, \psi_q \, \frac{\log q}{q^{\sigma+1}-1} = \lim_{i\to\infty} \sum_q \frac{\mathcal{N}_q(K)}{n_K} \, \frac{\log q}{q^{\sigma+1}-1}.
$$

\noindent
Thus, comparing \eqref{gamma-zeros} and \eqref{euler-2}, one may be tempted to believe that convergence of $\sum_q \psi_q \frac{\log q}{q}$
is intricately connected to the bounds on 
$$
    -\frac{\gamma_K}{ n_K}.
$$
Bounds on $\gamma_K$ have been extensively studied in  recent times. In \cite{Ihara1}, under GRH Ihara showed that
$$
    2\log \log \sqrt{d_K} \geq \gamma_K \geq -{\log \sqrt{d_K}}.
$$
Furthermore, a bound in terms of Siegel-zeros was obtained in the recent work of the first author and M. R. Murty in \cite{Dixit-Murty}. However, upper bounds on $- \, \gamma_K / n_K$ still remain a mystery. 

\bigskip

\subsection{\bf Final remarks}

The idea of using discriminants to study lower bounds on heights of algebraic numbers is not new. However, using explicit formula to connect this to the study of low lying zeros of the Dedekind zeta-function $\zeta_K(s)$ is the main idea of this paper. We also note that the terms $N_K(1)$ and $N_K(2)$ in our theorems counting the low-lying zeros up to height $1$ or $2$ can be replaced by $N_K(c)$ for any bounded real number $c$ to obtain similar lower bounds on the discriminant.\\ 

The study of towers of number fields draws inspiration from the theory of asymptotically exact families introduced by M. Tsfasman and S. G. Vl\u{a}du\c{t} in \cite{TV}. A family $\rK =\{K_i\}$ of number fields is said to be asymptotically exact if the following limits exist.
\begin{equation*}
\phi_q := \lim_{i\to\infty} \frac{\mathcal{N}_q(K_i)}{\log \sqrt{d_{K_i}}}
\end{equation*}
for all prime powers $q$ and
\begin{equation*}
\phi_{\mathbb{R}} := \lim_{i\to \infty} \frac{r_1(K_i)}{\log \sqrt{d_{K_i}}}, \hspace{5mm} \phi_{\mathbb{C}} := \lim_{i\to \infty} \frac{r_2(K_i)}{\log \sqrt{d_{K_i}}},
\end{equation*}
where $r_1(K_i)$ and $r_2(K_i)$ are the number of real and complex embeddings of $K_i$ respectively. We say that an asymptotically exact family $\mathcal{K} = \{K_i\}$ is \textit{asymptotically bad}, if $\phi_q = \phi_{\mathbb{R}} = \phi_{\mathbb{C}} =0$ for all prime powers $q$. This is analogous to saying that the root discriminant $ d_{K_i}^{1/n_{K_i}}$ tends to infinity as $i\to \infty$. If an asymptotically exact family $\mathcal{K}$ is not asymptotically bad, we say that it is \textit{asymptotically good}. By Minkowski's bound, we know that $\log d_K \geq c \, n_K$ where $c>0$ is an absolute constant. Thus, any asymptotically good tower with $\phi_q>0$ for at least one prime power $q$ is also asymptotically positive. Furthermore, by Theorem \ref{Discriminant_Bound_2}, for asymptotically good towers, the sum 
$$
\sum_q \psi_q \, \frac{\log q}{q+1} < \infty.
$$
This is in accordance with the prediction of Bombieri-Zannier that the sum in the RHS of \eqref{Bom-Zan-sum} converges.\\

As noted earlier, an attempt to obtain lower bounds on algebraic integers via lower bounds on the discriminant of a number field has three important components, namely, the index of $\alpha$, number of low-lying zeros of the Dedekind zeta-function and the splitting nature of primes (not necessarily complete splitting). It is desirable to have  a better understanding of how each of the factors influence the other. For instance, does the splitting nature of primes predict the density of low-lying zeros of the Dedekind zeta-function. At present, we are unaware of such relations and relegate it to future investigation.\\

On another front, this line of study naturally leads us to consider the elliptic analogue of Lehmer's problem. Let $E/K$ be an elliptic curve and $\widehat{h}:E(\overline{K})\to \R$ denote the canonical N\'{e}ron-Tate height. For a non-torsion point $P\in E(\overline{K})$, let $K(P)$ be the field of definition and $D(P) := [K(P):K]$. The elliptic analog of Lehmer's conjecture is formulated as
$$
    \widehat{h}(P) \geq \frac{c}{D(P)}
$$
for an absolute constant $c>0$. If we restrict ourselves to points in $K_{ab}$, one can prove a stronger result (Bogomolov property (B)), namely, $\widehat{h}(P) > c$. This was shown by J. Silverman in \cite{Silverman} (also see Baker \cite{Baker}, Hindry-Silverman \cite{Hindry}, Masser \cite{Masser}). However, Lehmer's conjecture still remains open in general. In view of Theorem \ref{Uncondional}, one may ask if Lehmer's conjecture can be established for certain infinite extensions where we have information about the splitting behaviour of certain primes. We relegate this to future research.

\medskip

\section{\bf Acknowledgement}
\medskip

We thank Prof. Kumar Murty, Prof. Ram Murty and Prof. Michel Waldschmidt for several fruitful discussions. We are especially indebted to Prof. Sinnou David for going through an earlier draft of this paper and providing crucial insights.  We express our appreciation to Dr. Siddhi Pathak for offering helpful comments on a previous version of this manuscript.

\end{document}